\documentclass[11pt]{article}
\usepackage{hyperref}
\usepackage{enumerate,fullpage}
\usepackage{amssymb,amsmath,graphicx,amsthm}

\usepackage{epsfig}
\usepackage{makeidx}
\usepackage{graphicx, epstopdf}
\usepackage{enumerate}
\usepackage{xcolor}

\makeindex
\newcommand{\ncom}{\newcommand}
\ncom{\ul}{\underline}
\ncom{\beq}{\begin{equation}}
\ncom{\eeq}{\end{equation}}
\ncom{\bea}{\begin{eqnarray*}}
\ncom{\eea}{\end{eqnarray*}}
\ncom{\beqa}{\begin{eqnarray}}
\ncom{\eeqa}{\end{eqnarray}}
\ncom{\nno}{\nonumber}
\ncom{\non}{\nonumber}
\ncom{\ds}{\displaystyle}
\ncom{\half}{\frac{1}{2}}
\ncom{\mbx}{\makebox{.25cm}}
\ncom{\hs}{\mbox{\hspace{.25cm}}}
\ncom{\rar}{\rightarrow}
\ncom{\Rar}{\Rightarrow}
\ncom{\noin}{\noindent}
\ncom{\bc}{\begin{center}}
\ncom{\ec}{\end{center}}
\ncom{\sz}{\scriptsize}
\ncom{\rf}{\ref}
\ncom{\s}{\sqrt{2}}
\ncom{\sgm}{\sigma}
\ncom{\Sgm}{\Sigma}
\ncom{\psgm}{\sigma^{\prime}}
\ncom{\dt}{\delta}
\ncom{\Dt}{\Delta}
\ncom{\lmd}{\lambda}
\ncom{\Lmd}{\Lambda}
\ncom{\Th}{\Theta}
\ncom{\e}{\eta}
\ncom{\eps}{\epsilon}
\ncom{\pcc}{\stackrel{P}{>}}
\ncom{\lp}{\stackrel{L_{p}}{>}}
\ncom{\dist}{{\rm\,dist}}
\ncom{\sspan}{{\rm\,span}}
\ncom{\re}{{\rm Re\,}}
\ncom{\im}{{\rm Im\,}}
\ncom{\sgn}{{\rm sgn\,}}
\ncom{\ba}{\begin{array}}
\ncom{\ea}{\end{array}}
\ncom{\hone}{\mbox{\hspace{1em}}}
\ncom{\htwo}{\mbox{\hspace{2em}}}
\ncom{\hthree}{\mbox{\hspace{3em}}}
\ncom{\hfour}{\mbox{\hspace{4em}}}
\ncom{\vone}{\vskip 2ex}
\ncom{\vtwo}{\vskip 4ex}
\ncom{\vonee}{\vskip 1.5ex}
\ncom{\vthree}{\vskip 6ex}
\ncom{\vfour}{\vspace*{8ex}}
\ncom{\norm}{\|\;\;\|}
\ncom{\integ}[4]{\int_{#1}^{#2}\,{#3}\,d{#4}}
\ncom{\vspan}[1]{{{\rm\,span}\{ #1 \}}}
\ncom{\dm}[1]{ {\displaystyle{#1} } }
\ncom{\ri}[1]{{#1} \index{#1}}

\newtheorem{theorem}{\bf Theorem}[section]
\newtheorem{remark}{\bf Remark}[section]
\newtheorem{proposition}{Proposition}[section]
\newtheorem{lemma}{Lemma}[section]
\newtheorem{corollary}{Corollary}[section]

\newtheorem{definition}{Definition}[section]
\newtheoremstyle
    {remarkstyle}
    {}
    {11pt}
    {}
    {}
    {\bfseries}
    {:}
    {     }
    {\thmname{#1} \thmnumber{#2} }

\theoremstyle{remarkstyle}

\begin{document}

\newpage

\begin{center}
{\Large \bf Skellam Type Processes of Order K and Beyond}
\end{center}
\vone
\begin{center}
{Neha Gupta}$^{\textrm{a}}$, {Arun Kumar}$^{\textrm{a}}$, {Nikolai Leonenko}$^{\textrm{b}}$
\footnotesize{
		$$\begin{tabular}{l}
		$^{\textrm{a}}$ \emph{Department of Mathematics, Indian Institute of Technology Ropar, Rupnagar, Punjab - 140001, India}\\
		
$^{b}$Cardiff School of Mathematics, Cardiff University, Senghennydd Road,
Cardiff, CF24 4AG, UK

\end{tabular}$$}
\end{center}
%\vtwo

\vtwo
\begin{center}
\noindent{\bf Abstract}
\end{center}
In this article, we introduce Skellam process of order $k$ and its running average. We also discuss the time-changed Skellam process of order $k$. In particular we discuss space-fractional Skellam process  and tempered space-fractional Skellam process via time changes in Poisson process by independent stable subordinator and tempered stable subordinator, respectively. We derive the marginal probabilities, L\'evy measures, governing difference-differential equations of the introduced processes. Our results generalize Skellam process and running average of Poisson process in several directions.\\

    
\section{Introduction}
Skellam distribution is obtained by taking the difference between two independent Poisson distributed random variables which was introduced for the case of different intensities  $\lambda_1,\; \lambda_2$ by (see \cite{Skellam1946}) and for equal means in \cite{Irwin1937}. For large values of $\lambda_1+\lambda_2$,  the distribution can be approximated by the normal distribution and if $\lambda_2$ is very close to $0$, then the distribution tends to a Poisson distribution with intensity $\lambda_1$. Similarly, if $\lambda_1$ tends to 0, the distribution tends to a Poisson distribution with non-positive integer values. The Skellam random variable is infinitely divisible since it is the difference of two infinitely divisible random variables (see Prop. $2.1$ in \cite{Steutel2004}). Therefore, one can define a continuous time L\'evy process for Skellam distribution which is called Skellam process.

\noindent The Skellam process is the integer valued L\'evy process and can also be obtained by taking the difference of two independent Poisson processes which marginal probability mass funcion (PMF) involves the modified Bessel function of the first kind. Skellam process has various applications in different areas such as to model the intensity difference of pixels in cameras (see  \cite{Hwang2007}) and for modeling the difference of the number of goals of two competing teams in football game in \cite{Karlis2008}.
The model based on the difference of two point processes are proposed in  (see \cite{Bacry2013a, Bacry2013b, Barndorff2011, Carr2011}).

\noindent Recently, time-fractional Skellam processes have studied in \cite{Kerss2014} which is obtained by time-changing the Skellam process with an inverse stable subordinator. Further, they provided the application of time-fractional Skellam process in modeling of arrivals of jumps in high frequency trading data. It is shown that the inter arrival times between the positive and negative jumps follow Mittag-Leffler distribution rather then the exponential distribution. Similar observations are observed in case of Danish fire insurance data (see \cite{Kumar2019}). Buchak and Sakhno in \cite{Buchak2018} also have proposed the governing equations for time-fractional Skellam processes. Recently,  \cite{Ayushi2020} introduced time-changed Poisson process of order $k$, which is obtained by time changing the Poisson process of order $k$ (see \cite{Kostadinova2012}) by general subordinators.

\noindent In this paper we introduce Skellam process of order $k$ and its running average. We also discuss the time-changed Skellam process of order $k$. In particular we discuss space-fractional Skellam process  and tempered space-fractional Skellam process via time changes in Poisson process by independent stable subordiantor and tempered stable subordiantor, respectively. We obtain closed form expressions for the marginal distributions of the considered processes and other important properties. Skellam process is used to model the difference between the number of goals between two teams in a football match. Similarly, Skellam process of order $k$ can be used to model the difference between the number of points scored by two competing teams in a basketball match where $k=3.$

The remainder of this paper proceeds as follows: in Section $2$, we introduce  all the relevant definitions and results. We derive also the L\'evy density for space- and tempered space-fractional Poisson processes. In Section $3$, we introduce and study running average of Poisson process of order $k$. Section 4 is dedicated to Skellam process of order $k$. Section 5 deals with running average of Skellam process of order $k$. In Section 6, we discuss about the time-changed Skellam process of order $k$. In Section $7$, we determine the marginal PMF, governing equations for marginal PMF, L\'evy  densities and moment generating functions for space-fractional Skellam process and tempered space-fractional Skellam process.
\section{Preliminaries}
In this section, we collect  relevant definitions and some results on Skellam process, subordinators, space-fractional Poisson process and tempered space-fractional Poisson process. These results will be used to define the space-fractional Skellam processes and tempered  space-fractional Skellam processes.

\subsection{Skellam process}

In this section, we revisit the Skellam process and also provide a characterization of it. Let $S(t)$ be a Skellam process, such that
$$  S(t)= N_{1}(t)-N_{2}(t), \; t\geq0,$$
where $N_{1}(t)$ and $N_{2}(t)$ are two independent homogeneous Poisson processes with intensity $\lambda_{1} >0$ and $\lambda_2>0,$ respectively. The Skellam process is defined in \cite{Barndorff2011} and the distribution has been introduced and studied in  \cite{Skellam1946}, see also \cite{Irwin1937}. This process is symmetric only when $\lambda_1= \lambda_2$.
The PMF $s_{k}(t)=\mathbb{P}(S(t)=k)$ of $S(t)$ is given by (see e.g. \cite{Skellam1946, Kerss2014})
\begin{align}{\label{Skellam_PMF}}
s_{k}(t)=e^{-t(\lambda_1+\lambda_2)}{\left(\frac{\lambda_1}{\lambda_2}\right)}^{k/2}I_{|k|}(2t\sqrt{\lambda_1 \lambda_2}),\; k\in \mathbb{Z},
\end{align}
where $I_k$ is modified Bessel function of first kind (see \cite{Abramowitz1974}, p. $375$),
\begin{equation}\label{Modi_Bessel}
I_{k}(z)=\sum_{n=0}^{\infty}\frac{{(z/2)}^{2n+k}}{n!(n+k)!}.
\end{equation}
The PMF $s_{k}(t)$ satisfies the following differential difference equation (see \cite{Kerss2014})
\begin{equation}
\frac{d}{dt}s_{k}(t)= \lambda_{1}(s_{k-1}(t)-s_{k}(t))-\lambda_{2}(s_{k}(t)-s_{k+1}(t)),\;\; k\in \mathbb{Z},
\end{equation}
with initial conditions $s_{0}(0)=1$ and $s_{k}(0)=0, \;k\neq0$. The Skellam process is a L\'evy process, its L\'evy density $\nu_S$ is the linear combination of two Dirac delta function, $\nu_S(y)= \lambda_{1}\delta_{\{1\}}(y)+\lambda_{2}\delta_{\{-1\}}(y) $ and the corresponding L\'evy exponent is given by
$$
\phi_{S(1)}(\theta)=\int_{-\infty}^{\infty}(1-e^{-\theta y})\nu(y)dy.
$$
The moment generating function (MGF) of Skellam process is 
\begin{align}
    \mathbb{E}[e^{\theta S(t)}]=e^{-t(\lambda_{1}+\lambda_{2}-\lambda_{1}e^{\theta}-\lambda_{2}e^{-\theta})}, \; \theta \in \mathbb{R}.    
\end{align}
With the help of MGF, one can easily find the moments of Skellam process. In next result, we give a characterization of Skellam process, which is not available in literature as per our knowledge.
\begin{theorem}
Suppose an arrival process has the independent and stationary increments and also satisfies the following incremental condition, then the process is Skellam.
\[   
\mathbb{P}(S(t+\delta)=m|S(t) = n)=
         \begin{cases}
                  \lambda_1 \delta + o(\delta), & m >n,\;m= n+1;\\
                 
                  \lambda_2 \delta + o(\delta), & m< n,\; m= n-1;\\
                   1-\lambda_1 \delta-\lambda_2 \delta + o(\delta), & m=n;\\
                   0 &\quad {\rm otherwise.}\\
           \end{cases}
\]
\end{theorem}
\begin{proof}
Consider the interval [0,t] which is discretized with $n$ sub-intervals of size $\delta$ each such that $n\delta =t.$ For $k\geq 0$, we have
\begin{align*}
\mathbb{P}(S(0,t) = k) &= \sum_{m=0}^{[\frac{n-k}{2}]} \frac{n!}{m!(m+k)!(n-2m-k)!}(\lambda_1\delta)^{m+k}(\lambda_2\delta)^{m}(1-\lambda_1\delta-\lambda_2\delta)^{n-2m-k} \\
&= \sum_{m=0}^{[\frac{n-k}{2}]} \frac{n!}{m!(m+k)!(n-2m-k)!} \left(\frac{\lambda_1 t}{n}\right)^{m+k}\left(\frac{\lambda_2 t}{n} \right)^{m}\left(1-\frac{\lambda_1 t}{n} -\frac{\lambda_2 t}{n}\right)^{n-2m-k}\\
&= \sum_{m=0}^{[\frac{n-k}{2}]} \frac{(\lambda_1 t)^{m+k}(\lambda_2 t)^{m}}{m!(m+k)!} \frac{n!}{(n-2m-k)!n^{2m+k}}\left(1-\frac{\lambda_1 t}{n} -\frac{\lambda_2 t}{n}\right)^{n-2m-k}\\
&= e^{-(\lambda_1+\lambda_2)t} \sum_{m=0}^{\infty}\frac{(\lambda_1 t)^{m+k}(\lambda_2 t)^{m}}{m!(m+k)!},
\end{align*}
by taking $n\rightarrow \infty.$ The result follows now by using the definition of modified Bessel function of first kind $I_k$. Similarly, we prove when $k<0.$
\end{proof}

\subsection{Poisson process of order $k$ (PPoK)} 

In this section, we recall the definition and some important properties of Poisson process of order k (PPoK). Kostadinova and Minkova  (see \cite{Kostadinova2012}) introduced and studied the PPok. Let $ x_1, x_2, \cdots, x_k $ be non-negative integers and $\zeta_{k}  = x_1 + x_2 + \dots + x_k, \; \Pi_{k}! = x_1!x_2!\dots x_k! $ and
\begin{equation}
    \Omega(k,n) = \{ X=(x_1,\ x_2,\ \dots,\ x_k) | x_1 + 2x_2+ \dots + kx_k=n\}.
\end{equation}
Also, let $\{N^{k}(t)\}_{t\geq 0},$ represent the PPok with rate parameter $\lambda t$, then probability mass function (pmf) is given by
\begin{equation}\label{pmf_ppok}
    p_{n}^{N^{k}}(t)=\mathbb{P}(N^{k}(t) = n) = \sum_{X=\Omega(k,n)} e^{-k\lambda t} \frac{(\lambda t)^{\zeta_{k}}}{\Pi_{k}!}.
\end{equation}
The pmf of $N^{k}(t)$ satisfies the following differential-difference equations (see \cite{Kostadinova2012})
\begin{align}\label{SFPP}
 \frac{d}{dt}p_{n}^{N^{k}}(t) &= -k \lambda p_{n}^{N^{k}}(t)+\lambda\sum_{j=1}^{n\wedge k}p_{n-k}^{N^{k}}(t),\;\; n=1,2,\ldots \nonumber \\
\frac{d}{dt}p_{0}^{N^{k}}(t)& = -k \lambda p_{0}^{N^{k}}(t),
\end{align}
with initial condition $p_{0}^{N^{k}}(0) = 1$ and $p_{n}^{N^{k}}(0) = 0$ and $n\wedge k = \min\{k,n\}$. The characteristic function of PPoK $N^{k}(t)$
 \begin{equation}\label{char_ppok}
     \phi_{N^{k}(t)}(u)= \mathbb{E}(e^{iuN^{k}(t)})= e^{- \lambda t (k- \sum_{j=1}^{k}e^{iuj})},
 \end{equation}
 where $i=\sqrt{-1}$. The process PPoK is L\'evy, so it is infinite divisible i.e. $\phi_{N^{k}(t)}(u) = (\phi_{N^{k}(1)}(u))^{t}.$ The L\'evy measure for PPoK is easy to drive and is given by
$$
\nu_{N^k}(x) = \lambda \sum_{j=1}^{k}\delta_j(x),
$$
where $\delta_j$ is the Dirac delta function concentrated at $j$.
 The transition probability  of the PPoK $\{N^{k}(t)\}_{t\geq 0}$ are also given by Kostadinova and Minkova \cite{Kostadinova2012},
\begin{equation}
\mathbb{P}(N^{k}(t+\delta)=m|N^{k}(t) = n)=
         \begin{cases}
                  1-k\lambda \delta, & m=n;\\
                  \lambda \delta & m= n+i, i=1,2, \ldots,k;\\
                  0 & \quad{\rm otherwise.}\\
           \end{cases}
\end{equation}
The probability generating function (pgf) $G^{N^{k}}(s,t)$  is given by (see \cite{Kostadinova2012})  
\begin{equation}\label{pgf_ppok}
   G^{N^{k}}(s,t) = e^{-\lambda t(k-\sum_{j=1}^{k} s^{j})}.
\end{equation}
The mean, variance and covariance function of the PPoK are given by
\begin{align}
    \mathbb{E}[N^{k}(t)] &= \frac{k(k+1)}{2}\lambda t; \nonumber \\
    {\rm Var}[N^{k}(t)]& = \frac{k(k+1)(2k+1)}{6}\lambda t; \nonumber\\
    {\rm Cov}[N^{k}(t), N^{k}(s)]& =  \frac{k(k+1)(2k+1)}{6} \lambda (t\wedge s).
\end{align}

\subsection{Subordinators}
Let $D_{f}(t)$ be real valued  L\'evy process with non-decreasing sample paths and its Laplace transform has the form
$$ \mathbb{E}[e^{-s D_{f}(t)}] = e^{-tf(s)}, $$
where
$$ f(s) = bs + \int_{0}^{\infty} (1-e^{x s}) \nu(dx),\;\; s>0,\; b\geq0,$$
is the integral representation of  Bernstein functions (see \cite{Schilling2010}).
The Bernstein functions are $C^{\infty}$, non-negative and such that $(-1)^{m}\frac{d^{m}}{dx^m}f(x) \leq 0$ for $m\geq 1$ in \cite{Schilling2010}. Here $\nu$ denote the non-negative L\'evy measure on the positive half line such that
$$ \int_{0}^{\infty}(x \wedge 1)  \nu(dx) < \infty,  \; \; \nu([0, \infty)) = \infty,
$$
and b is the drift coefficient. The right continuous inverse $E_{f}(t) = \inf\{u\geq0 : D_{f}(u) >t\}$ is the inverse and first exist time of $D_{f}(t)$, which is non-Markovian with non-stationary and non-independent increments. Next, we analyze some special cases of L\'evy subordinators with drift coefficient b = 0, that is,
\begin{equation}\label{Levy_exponent}
f(s) =
         \begin{cases}
                  p\log(1+\frac{s}{\alpha}),\; p>0,\ \alpha >0, &$(gamma subordinator)$;\\
                  (s+\mu)^{\alpha}-\mu^{\alpha},\;\mu >0,\; 0<\alpha<1,&$(tempered $\alpha$-stable  subordinator)$;\\
                   \delta(\sqrt{2s+\gamma^2}-\gamma),\; \gamma>0,\; \delta>0, & $(inverse Gaussian subordinator)$;\\
									s^{\alpha},\; 0<\alpha<1, & $(stable subordinator)$.
           \end{cases}
\end{equation}
It is worth to note that among the subordiantors given in \eqref{Levy_exponent}, all the integer order moments of stable subordiantor are infinite.

\subsection{The space-fractional Poisson process}
In this section, we discuss main properties of space-fractional Poisson process (SFPP). We also provide the L\'evy density for SFPP which is not discussed in the literature. The SFPP $N_{\alpha}(t)$ was introduced by (see \cite{Orsingher2012}), as follows
\begin{equation}
 N_{\alpha}(t) =
          \begin{cases}
                  N(D_{\alpha}(t)),\;t\geq 0, & 0<\alpha<1,\\
                  N(t),\;t\geq0, & \alpha=1.
           \end{cases}
\end{equation}
The probability generating function (PGF) of this process is of the form
\begin{align}\label{PGF of space}
G^{\alpha}(u,t)=\mathbb{E}u^{N_{\alpha}(t)}=e^{{-\lambda}^{\alpha}(1-u)^{\alpha}t},\; |u|\leq1, \; \alpha\in(0,1).
\end{align}  
The PMF of SFPP  is 
\begin{align}\label{space-fractional-PMF}
P^{\alpha}(k,t) =\mathbb{P}\{N_{\alpha}(t)=k\} &=\frac{(-1)^k}{k!}\sum_{r=0}^{\infty}\frac{(-\lambda^\alpha)^{r} t^r}{r!}\frac{\Gamma(r\alpha+1)}{\Gamma(r\alpha-k+1)}\nonumber\\
& =\frac{(-1)^k}{k!} {}_{1}\psi_{1} \left[\begin{matrix}
 (1, \alpha); \\
  (1-k, \alpha); 
 \end{matrix}(-\lambda^\alpha t) \right],
\end{align}
where ${}_{h}\psi_{i}(z)$ is the Fox Wright function (see formula $(1.11.14)$ in \cite{Kilbas2006}).
It was shown in \cite{Orsingher2012} that the PMF of the SFPP satisfies the following fractional differential-difference equations
\begin{align}\label{SFPP}
 \frac{d}{dt}P^{\alpha}(k,t) &= -\lambda^\alpha (1-B)^\alpha P^{\alpha}(k,t),\;\;  \alpha\in(0,1],\; k=1,2,\ldots \\
\frac{d}{dt}P^{\alpha}(0,t)& = -\lambda^\alpha P^{\alpha}(0,t),
\end{align}
with initial conditions            
\begin{equation}\label{intial condition}
P(k,0) =\delta_{k}(0)=
          \begin{cases}
                  0, & k\neq0,\\
                  1, & k=0.
           \end{cases}
\end{equation}
The fractional difference operator 
\begin{equation}\label{Backward_Operator}
(1-B)^\alpha = \sum_{j=0}^{\infty}{\alpha \choose j}(-1)^jB^{j}
\end{equation}
is defined in \cite{Beran1994}, where $B$ is the backward shift operator. The characteristic function of SFPP is
\begin{equation}\label{charc_SFPP}
 \mathbb{E}[e^{i\theta N_{\alpha}(t)}] = e^{-\lambda^{\alpha}(1-e^{i\theta})^{\alpha}t}.
\end{equation}

\begin{proposition}
The L\'evy density $\nu_{N_{\alpha}}(x)$ of SFPP is given by
\begin{equation}\label{space_levy}
\nu_{N_{\alpha}}(x) = \lambda^{\alpha}\sum^{\infty}_{n=1}(-1)^{n+1} {\alpha \choose n} \delta_{n}(x).
\end{equation}
\end{proposition}
\begin{proof} We use L\'evy-Khintchine formula (see \cite{Sato1999}),
\begin{align*}
\int_{{\{ 0 \}}^{c}}&(e^{i\theta x}-1)\lambda^{\alpha}\sum^{\infty}_{n=1}(-1)^{n+1} {\alpha \choose n} \delta_{n}(x)  dx \\
&=\lambda^{\alpha}\left[ \sum^{\infty}_{n=1}(-1)^{n+1} {\alpha \choose n} e^{i\theta n}+\sum_{n=0}^{\infty}(-1)^{n} {\alpha \choose n} -1 \right]\\
&= \lambda^{\alpha} \sum^{\infty}_{n=0}(-1)^{n+1} {\alpha \choose n} e^{i\theta n} =  -\lambda^{\alpha}(1-e^{i\theta})^{\alpha},
\end{align*}
which is the characteristic exponent of SFPP from equation \eqref{charc_SFPP}.
\end{proof}

\subsection{Tempered space-fractional Poisson  process}
The tempered space-fractional Poisson process (TSFPP) can be obtained by subordinating homogeneous Poisson process $N(t)$ with the independent tempered stable subordiantor $D_{\alpha, \mu}(t)$ (see \cite{Gupta2020})
\begin{equation}
N_{\alpha, \mu}(t) = N(D_{\alpha, \mu}(t)),\; \alpha \in(0,1),\; \mu > 0.
\end{equation}
This process have finite integer order moments due to the tempered $\alpha$-stable subordinator. The PMF of TSFPP is given by (see \cite{Gupta2020})
\begin{align}{\label{pmf_tem_space}}
P^{\alpha, \mu}(k,t) & = (-1)^k e^{t\mu^{\alpha}} \sum_{m=0}^{\infty}\mu^m\sum_{r=0}^{\infty}\frac{(-t)^r}{r!}\lambda^{\alpha r-m}{\alpha r \choose m} {\alpha r -m \choose k}\nonumber\\
&=e^{t\mu^{\alpha}}\frac{(-1)^k}{k!}\sum_{m=0}^{\infty}\frac{\mu^m \lambda^{-m}}{m!} {}_{1}\psi_{1} \left[\begin{matrix}
 (1, \alpha); \\
  (1-k-m, \alpha); 
 \end{matrix}(-\lambda^\alpha t) \right]
,\;k=0,1,\ldots.
\end{align}
The governing difference-differential equation is given by
\begin{equation}\label{tempered-SFPP}
\frac{d}{dt}P^{\alpha, \mu}(k,t) = - ((\mu + \lambda(1-B))^{\alpha} - \mu^{\alpha}) P^{\alpha, \mu}(k,t).
\end{equation}
The characteristic function of TSFPP,
\begin{equation}\label{tempe_char}
\mathbb{E}[e^{i\theta N_{\alpha, \mu}(t) }] = e^{- t((\mu + \lambda(1-e^{i\theta}))^{\alpha} - \mu^{\alpha})}.
\end{equation}
Using a standard conditioning argument, the mean and variance of TSFPP are given by
\begin{align}\label{vari_tem}
\mathbb{E}(N_{\alpha,\mu}(t)) =  \lambda \alpha \mu^{\alpha-1}t,\;\;
{\rm Var}(N_{\alpha,\mu}(t)) = \lambda \alpha \mu^{\alpha-1}t + \lambda^2 \alpha(1-\alpha) \mu^{\alpha-2}t.
\end{align}
\begin{proposition}
The L\'evy density $\nu_{N_{\alpha, \mu}}(x)$ of TSFPP is
\begin{equation}\label{levy_temp}
\nu_{N_{\alpha, \mu}}(x) = \sum_{n=1}^{\infty}\mu^{\alpha-n}{\alpha \choose n}\lambda^{n} \sum_{l=1}^{n} {n \choose l}(-1)^{l+1} \delta_{l}(x),\; \mu >  0.
\end{equation}
\end{proposition}
\begin{proof} Using \eqref{tempe_char}, the characteristic exponent of TSFPP is given by $F(\theta)=((\mu + \lambda(1-e^{i\theta}))^{\alpha} - \mu^{\alpha})$ .  We find the L\'evy density with the help of  L\'evy-Khintchine formula (see \cite{Sato1999}),
\begin{align*}
\int_{{\{ 0 \}}^{c}}&(e^{i\theta x}-1)\sum_{n=1}^{\infty}\mu^{\alpha-n}{\alpha \choose n}\lambda^{n} \sum_{l=1}^{n} {n \choose l}(-1)^{l+1} \delta_{l}(x) dx \\
&=\sum_{n=1}^{\infty}\mu^{\alpha-n}{\alpha \choose n}\lambda^{n} \left(\sum_{l=1}^{n} {n \choose l}(-1)^{l+1}e^{i\theta x}-\sum_{l=1}^{n} {n \choose l}(-1)^{l+1} \right)\\
&= \sum_{n=0}^{\infty}\mu^{\alpha-n}{\alpha \choose n}\lambda^{n} \sum_{l=0}^{n} {n \choose l}(-1)^{l+1} \delta_{l}(x)- \mu^{\alpha}\\
& =  -((\mu + \lambda(1-e^{i\theta}))^{\alpha} - \mu^{\alpha}),
\end{align*}
hence proved.
\end{proof}

\section{Running average of PPoK}
In this section, first we introduced the running average of PPoK and their main properties. These results will be used further to discuss the running average of SPoK. 
\begin{definition}[Running average of PPoK] 
We define the average process $N^{k}_{A}(t)$ by taking time-scaled integral of the path of the PPoK,
\begin{align}
    N^{k}_{A}(t) = \frac{1}{t}\int_{0}^{t}{N^{k}(s) ds}.
\end{align}
\end{definition}
\noindent We can write the differential equation with initial condition $N^{k}_{A}(0) =0$,
$$
\frac{d}{dt}(N^{k}_{A}(t))=\frac{1}{t}N^{k}(t) - \frac{1}{t^2}\int_{0}^{t}{N^{k}(s) ds}.
$$
Which shows that it has continuous sample paths of bounded total variation. We explored the compound Poisson representation and distribution properties of running average of PPoK. The characteristic of $N^{k}_{A}(t)$ is obtained by using the Lemma 1 of (see \cite{Xia2018}).
\begin{lemma}
If $X_{t}$ is a L\'evy process and $Y_{t}$ its Riemann integral defined by
\begin{align*}
    Y_{t} = \int_{0}^{t}{X_{s} ds},
\end{align*}
then the characteristic functions of $X$ and $Y$ satisfy
\begin{align}
    \phi_{Y(t)}(u) = \mathbb{E}[e^{iuY(t)}] = \exp\left( t \int_{0}^{1}\log{\phi_{X(1)}(tuz)} dz\right),\; u\in \mathbb{R}.
\end{align}
\end{lemma}

\begin{proposition} The characteristic function of $N_A^k(t)$ is given by
\begin{align}\label{cf_ppok}
    \phi_{N^{k}_{A}(t)}(u) = \exp\left(-t \lambda \left(k- \sum_{j=1}^{k}\frac{(e^{iuj}-1)}{iuj}\right)\right).
\end{align}
\end{proposition}
\begin{proof}
The result follows by applying the Lemma $1$ to \eqref{char_ppok} after scaling by $1/t$.
\end{proof}

\begin{proposition}\label{running_average_cp}
The running average process has a compound Poisson representation, such that
\begin{align}
    Y(t) = \sum_{i=1}^{N(t)}X_{i},
\end{align}
where $X_i = 1,2, \ldots$ are independent, identically distributed (iid) copies of $X$ random
variables, independent of $N(t)$ and $N(t)$ is a Poisson process with intensity $k \lambda$. Then
$$ Y(t) \stackrel{law}{=} N^{k}_{A}(t).$$ Further, the random variable $X$ has following pdf
\begin{equation}\label{pmf_X}
    f_{X}(x)= \sum_{i=1}^{k}p_{V_{i}}(x) f_{U_{i}}(x) = \frac{1}{k}\sum_{i=1}^{k}f_{U_{i}}(x),
\end{equation}
where $V_{i}$ follows discrete uniform distribution over $(0, k)$ and $U_{i}$ follows continuous uniform distribution over $(0,i),\;i=1,2,\ldots,k.$
\end{proposition}
\begin{proof}
The pdf of $U_i$ is $f_{U_{i}}(x) = \frac{1}{i} , \; 0 \leq x \leq i.$ Using \eqref{pmf_X}, the characteristic function of $X$ is given by
$$\phi_X(u) = \frac{1}{k}\sum_{j=1}^{k}\frac{(e^{iuj}-1)}{iuj}. $$
For fixed $t$, the characteristic functinon of $Y(t)$ is
\begin{equation}
    \phi_{Y(t)}(u)=  e^{-k \lambda t (1-\phi_{X}(u))} = e^{-t \lambda \left(k- \sum_{j=1}^{k}\frac{(e^{iuj}-1)}{iuj}\right)},
\end{equation}
which is equal to the characteristic function of PPoK given in \eqref{cf_ppok}. Hence by uniqueness of characteristic function the result follows.
\end{proof}

\noindent Using the definition
\begin{equation}\label{M_char}
   m_{r}= \mathbb{E}[X^{r}] = (-i)^r\frac{d^{r}\phi_{X}(u)}{du^{r}},
\end{equation}
the first two moments for random variable $X$ given in Proposition \eqref{running_average_cp} are
$m_{1} = \frac{(k+1)}{4}$ and $m_2= \frac{1}{18}[(k+1)(2k+1)]$.
Further using the mean, variance  and covariance of compound Poisson process, we have
\begin{align*}
    \mathbb{E}[N^{k}_{A}(t)]&=\mathbb{E}[N(t)]\mathbb{E}[X]=\frac{k(k+1)}{4}\lambda t;\\
    {\rm Var}[N^{k}_{A}(t)]& = \mathbb{E}[N(t)]\mathbb{E}[X^2]=\frac{1}{18}[k(k+1)(2k+1)]\lambda t;\\
    {\rm Cov}[N^{k}_{A}(t),N^{k}_{A}(s)] &= \mathbb{E}[N^{k}_{A}(t),N^{k}_{A}(s)] -\mathbb{E}[N^{k}_{A}(t)]\mathbb{E}[N^{k}_{A}(s)]\\
    &=\mathbb{E}[N^{k}_{A}(s)]\mathbb{E}[N^{k}_{A}(t-s)]-\mathbb{E}[N^{k}_{A}(s)^{2}]-\mathbb{E}[N^{k}_{A}(t)]\mathbb{E}[N^{k}_{A}(s)]\\
    & = \frac{1}{18}[k(k+1)(2k+1)]\lambda s -\frac{k^2(k+1)^{2}}{16}\lambda^2 s^2, \; s<t.
\end{align*}
\begin{remark}
Putting $k=1$, the running average of PPoK $N^{k}_{A}(t)$ reduces to running average of standard Poisson process $N_{A}(t)$ (see Appendix in \cite{Xia2018}).
\end{remark}
\begin{remark}
The mean and variance of PPoK and running average of PPoK satisfy,
$\mathbb{E}[N^{k}_{A}(t)]/ \mathbb{E}[N^{k}(t)] = \frac{1}{2}$ and $\mathrm{Var}[N^{k}_{A}(t)]/\mathrm{Var}[N^{k}(t)] = \frac{1}{3}$.
\end{remark}

\noindent Next we discuss the long-range dependence (LRD) property of running average of PPoK. We recall the definition of LRD for a non-stationary process.
\begin{definition}[Long range dependence (LRD)]\label{LRD_definition}
Let $X(t)$ be a stochastic process which has correlation function for $s \geq t$ for fixed $s$, that satisfies,
$$ c_{1}(s) t^{-d} \leq {\rm Cor}(X(t),X(s)) \leq c_{2}(s) t^{-d},$$
for large $t$, $d > 0$, $c_{1}(s) > 0$ and $c_2(s) > 0$. That is,
$$ \lim_{t\to\infty} \frac{{\rm Cor}(X(t),X(s))}{t^{-d}} = c(s) $$
for some $c(s) > 0$ and $d > 0$. We say that if $d \in (0, 1)$ then X(t) has the LRD property and if $d \in (1, 2)$ it has short-range dependence (SRD) property \cite{Maheshwari2016}.
\end{definition}
\begin{proposition}
The running average of PPoK has LRD property.
\end{proposition}
\begin{proof}
Let $0 \leq s< t < \infty$, then the correlation function for running average of PPoK $N_A^{k}(t)$ is
\begin{align*}
{\rm Cor}(N^{k}_{A}(t), N^{k}_{A}(s)) &=\frac{\left(8(2k+1)-9(k+1)k\lambda s\right)s^{1/2}t^{-1/2}}{8(2k+1)}.
\end{align*}
Then for $d=1/2$, it follows
\begin{align*}
\lim_{t\to\infty} \frac{{\rm Cor}(N^{k}_{A}(t), N^{k}_{A}(s))}{t^{-d}}& = \frac{\left(8(2k+1)-9(k+1)k \lambda s\right)s^{1/2}}{8(2k+1)}= c(s).
\end{align*}
\end{proof}

\section{Skellam process of order $K$ (SPoK)}
In this section, we introduce and study Skellam process of order $K$ (SPoK).
\begin{definition}[SPoK]
Let $ N^{k}_{1}(t)$ and $N^{k}_{2}(t)$ be two independent  PPoK with intensities $ \lambda_{1} >0$ and $ \lambda_{2} >0$. The stochastic process 
$$ S^{k}(t)= N^{k}_{1}(t) - N^{k}_{2}(t) $$
is called a Skellam process of order $K$ (SPoK).
\end{definition}
\begin{proposition}
The marginal distribution $R_{m}(t)=\mathbb{P}(S^{k}(t)=m)$ of SPoK $S^{k}(t)$ is given by 
\begin{equation}\label{pmf_spok}
    R_{m}(t)=e^{-kt(\lambda_1+\lambda_2)}{\left(\frac{\lambda_1}{\lambda_2}\right)}^{m/2}I_{|m|}(2tk\sqrt{\lambda_1 \lambda_2}),\; m\in \mathbb{Z}.
\end{equation}
\end{proposition}
\begin{proof}  For $m\geq 0$, using the pmf of PPoK given in \eqref{pmf_ppok}, it follows
  \begin{align*}
   R_{m}(t) &=  \sum^{\infty}_{n=0}\mathbb{P}(N^{k}_{1}(t)=n+m)\mathbb{P}(N^{k}_{2}(t)=n)\mathbb{I}_{m\geq 0}\\
       &= \sum_{n=0}^{\infty}\left(  \sum_{X=\Omega(k,n+m)} e^{-k\lambda_1 t} \frac{(\lambda_1 t)^{\zeta_{k}}}{\Pi_{k}!}\right)\left(  \sum_{X=\Omega(k,n)} e^{-k\lambda_2 t} \frac{(\lambda_2 t)^{\zeta_{k}}}{\Pi_{k}!}\right)\\
  \end{align*}
Setting $x_{i}= n_{i}$ and $n=x+\sum_{i=1}^{k}(i-1)n_{i}$, we have
\begin{align*}
    R_{m}(t) &=e^{-kt(\lambda_{1} +\lambda{2})}\sum_{x=0}^{\infty}\frac{(\lambda_2 t)^ {x}}{x!}\frac{(\lambda_1 t)^{m+x}}{(m+x)!}\left(\sum_{n_{1}+n_{2}+\ldots +n_{k}=m+x} {m+x \choose n_1!n_2!\ldots n_k!}\right)\left(\sum_{n_{1}+n_{2}+\ldots +n_{k}=x} {x \choose n_1!n_2!\ldots n_k!}\right)\\
    &= e^{-kt(\lambda_{1} +\lambda{2})}\sum_{x=0}^{\infty}\frac{(\lambda_2 t)^ {x}}{x!}\frac{(\lambda_1 t)^{m+x}}{(m+x)!} k^{m+x}k^{x},
\end{align*}
using the multinomial theorem and modified Bessel function given in \eqref{Modi_Bessel}. Similarly, it follows for $m<0$.
\end{proof}

\noindent In the next proposition, we prove the normalizing condition for SPoK.
\begin{proposition}
The pmf of $S^{k}(t)$ satisfies the following normalizing condition
$$ \sum_{m=-\infty}^{\infty} R_{m}(t) = 1.$$
\end{proposition}
\begin{proof}
Using the property of modified Bessel function of first kind 
  $$ \sum_{y=-\infty}^{\infty} \left(\frac{\theta_1}{\theta_2}\right)^{y/2}I_{|m|}(2a\sqrt{\theta_1 \theta_2}) = e^{a(\theta_1 + \theta_2 )},$$
 and puting this result in \eqref{pmf_spok}, we obtain 
 $$\sum_{m=-\infty}^{\infty} R_{m}(t) = e^{-kt(\lambda_1+\lambda_2)}e^{kt(\lambda_1+\lambda_2)} =1.$$
\end{proof}

\begin{proposition}
The L\'evy measure for SPoK is
$$
\nu_{S^k}(x) = \lambda_1 \sum_{j=1}^{k}\delta_j(x) + \lambda_2 \sum_{j=1}^{k}\delta_{-j}(x).
$$
\end{proposition}
\begin{proof}
The proof follows by using the independence of two PPoK used in the definition of SPoK.
\end{proof}

\begin{remark}
Using \eqref{pgf_ppok}, the pgf of SPoK is given by
\begin{equation}\label{pgf_spok}
   \displaystyle G^{S^{k}}(s,t) = \sum^{\infty}_{m=-\infty} s^{m}R_{m}(t)=  e^{-t\left(k(\lambda_1 + \lambda_2) -\lambda_1\sum_{j=1}^{k}s^{j} -\lambda_2 \sum_{j=1}^{k}s^{-j}\right)}.
\end{equation}
\noindent Further, the characteristic function of SPoK is given by 
\begin{equation}\label{char_spok}
    \phi_{S^{k}(t)}(u) = e^{-t[k(\lambda_1 + \lambda_2) -\lambda_1\sum_{j=1}^{k}e^{iju} -\lambda_2 \sum_{j=1}^{k}e^{-iju}]}.
\end{equation}
\end{remark}
\subsection{SPoK as a pure birth and death process}
In this section, we provide the transition probabilities of SPoK at time $t + \delta$, given that we started at time $t$. Over such a short interval of length $\delta \rightarrow 0$, it is nearly impossible to observe more than $k$ event; in fact, the probability to see more than $k$ event is $o(\delta)$.
\begin{proposition}
The transition probabilities of SPoK are given by
\begin{equation}
\mathbb{P}(S^{k}(t+\delta)=m|S^{k}(t) = n)=
         \begin{cases}
                  \lambda_1 \delta + o(\delta), & m >n,\;m= n+i, i=1,2, \ldots,k;\\
                 
                  \lambda_2 \delta + o(\delta), & m< n,\; m= n-i, i=1,2, \ldots,k;\\
                   1-k\lambda_1 \delta-k\lambda_2 \delta + o(\delta), & m=n.\\
           \end{cases}
\end{equation}
Basically, at most k events can occur in a very small interval of time $\delta$. And even though the probability for more than k event is non-zero, it is negligible.
\end{proposition}
\begin{proof} Note that for $i=1,2,\cdots,k$, we have
\begin{align*}
\mathbb{P}(S^{k}(t+\delta)= n+i|S^{k}(t) = n) &= \sum_{j=1}^{k-i}\mathbb{P}(\mbox{the first process has i+j arrivals and the second process has j arrivals})\\ 
& + \mathbb{P}(\mbox{the first process has i arrivals and the second process has 0 arrivals})\\
& = \sum_{j=0}^{k-i} (\lambda_1\delta + o(\delta)) \times  (\lambda_2\delta + o(\delta)) + (\lambda_1\delta + o(\delta)) \times (1-k\lambda_2\delta + o(\delta)) \\
& = \lambda_1\delta + o(\delta).
\end{align*}
Similarly, for $i=1,2,\cdots,k$, we have
\begin{align*}
\mathbb{P}(S^{k}(t+\delta)= n-i|S^{k}(t) = n) &= \sum_{j=1}^{k-i}\mathbb{P}(\mbox{the first process has j arrivals and the second process has i+j arrivals})\\ 
& + \mathbb{P}(\mbox{the first process has 0 arrivals and the second process has i arrivals})\\
& = \sum_{j=0}^{k-i} (\lambda_1\delta + o(\delta)) \times  (\lambda_2\delta + o(\delta)) + (1-k\lambda_1\delta + o(\delta)) \times (\lambda_2\delta + o(\delta)) \\
& = \lambda_2\delta + o(\delta).
\end{align*}
Further,
\begin{align*}
\mathbb{P}(S^{k}(t+\delta)= n|S^{k}(t) = n) &= \sum_{j=1}^{k}\mathbb{P}(\mbox{the first process has j arrivals and the second process has j arrivals})\\ 
& + \mathbb{P}(\mbox{the first process has 0 arrivals and the second process has 0 arrivals})\\
& = \sum_{j=0}^{k} (\lambda_1\delta + o(\delta)) \times  (\lambda_2\delta + o(\delta)) + (1-k\lambda_1\delta + o(\delta)) \times (1-k\lambda_2\delta + o(\delta)) \\
& = 1-k\lambda_1\delta - k\lambda_2\delta + o(\delta).
\end{align*}
\end{proof}

\begin{remark}
The pmf $R_{m}(t)$ of SPoK satisfies the following difference differential equation 
\begin{align*}\label{diff_spok}
    \frac{d}{dt} R_{m}(t) &= -k(\lambda_{1}+\lambda_2)R_{m}(t)+\lambda_{1} \sum_{j=1}^{k} R_{m-j}(t) + \lambda_{2} \sum_{j=1}^{k} R_{m+j}(t)\\
    &= -\lambda_{1}\sum_{j=1}^{k}(1-B^{j}) R_{m} - \lambda_{2} \sum_{j=1}^{k}(1-F^{j})R_{m}(t),\;\; m\in \mathbb{Z},
\end{align*}
with initial condition $R_{0}(0) = 1$ and $R_{m}(0) = 0$ for $m \neq 0$,Let $B$ be the backward shift operator defined in \eqref{Backward_Operator} and $F$ be the forward shift operator defined by
 $F^jX(t) = X(t+j)$ such that $(1-F)^\alpha = \sum_{j=0}^{\infty}{\alpha \choose j}F^{j}$. Multiplying by  $s^{m}$ and summing for all $m$ in \eqref{diff_spok}, we get the following differential equation for the pgf
$$ \frac{d}{dt}G^{S^{k}}(s,t) = \left(-k(\lambda_1 + \lambda_2) +\lambda_1\sum_{j=1}^{k}s^{j} +\lambda_2 \sum_{j=1}^{k}s^{-j}\right)G^{S^{k}}(s,t).
$$
%The solution of above differential equation with the initial condition  $G^{S^{k}}(0,t) =1$ is the pgf of SPoK $S^{k}(t)$.
\end{remark}
\noindent The mean, variance and covariance of SPoK can be easily calculated by using the pgf,
\begin{align*}
    \mathbb{E}[S^{k}(t)]& =\frac{k(k+1)}{2}(\lambda_1 -\lambda_2) t;\\
    {\rm Var}[S^{k}(t)]& =\frac{1}{6}\left[k(k+1)(2k+1)\right](\lambda_{1} +\lambda_{2}) t;\\
    {\rm Cov}[S^{k}(t),S^{k}(s)] &= \frac{1}{6}\left[k(k+1)(2k+1)\right](\lambda_1 +\lambda_2)s, \;\; s<t.
\end{align*}

%\subsection{Long-range dependency of SPoK} 
\noindent Next we show the LRD property for SPoK.
\begin{proposition}
The SPoK has LRD property defined in Definition \ref{LRD_definition}.
\end{proposition}
\begin{proof}
The correlation function of SPoK satisfies
$$\lim_{t\to\infty} \frac{{\rm Cor}(S^k(t), S^k(s))}{t^{-d}} = \frac{s^{1/2} t^{-1/2}}{t^{-1/2}} =c(s).$$
Hence SPoK exhibits the LRD property.
\end{proof}

\section{Running average of SPoK}
In this section, we introduce and study the new stochastic L\'evy process which is running average of SPoK. 
\begin{definition}
The following  stochastic process defined by taking time-scaled integral of the path of the SPoK,
\begin{align}
    S^{k}_{A}(t) = \frac{1}{t}\int_{0}^{t}{S^{k}(s) ds},
\end{align} 
is called the running average of SPoK.
\end{definition}
\noindent Next  we provide the compound Poisson representation of running average of SPoK.
\begin{proposition} The characteristic function $\phi_{S^{k}_{A}(t)}(u) = \mathbb{E}[e^{iu S^{k}_{A}(t)}]$ of $S^{k}_{A}(t)$ is given by
\begin{align}\label{char_raspok}
    \phi_{S^k_A(t)}(u) = \exp\left[-kt \left\{ \lambda_1 \left(1- \frac{1}{k}\sum_{j=1}^{k}\frac{(e^{iuj}-1)}{iuj} \right)+\lambda_2\left(1- \sum_{j=1}^{k}\frac{(1-e^{-iuj})}{iuj}\right) \right\} \right],\; u\in \mathbb{R}.
\end{align}
\end{proposition}
\begin{proof}
By using the Lemma $1$ to equation \eqref{char_spok} after scaling by $1/t$.
\end{proof}
\begin{remark}
It is easily observable that in equation \eqref{char_raspok} has removable singularity at $u=0$. To remove that singularity we can define $\phi_{S^{k}_{A}(t)}(0) =1$.
\end{remark}
\begin{proposition} Let $Y(t)$ be a compound Poisson process 
\begin{align}
    Y(t) = \sum_{n=1}^{N(t)}J_{n}, 
\end{align}
where $N(t)$ is a Poisson process with rate parameter $k(\lambda_1 + \lambda_2) >0$ and $\{J_{n}\}_{n \geq 1}$ are iid random variables with mixed double uniform distribution function $p_{j}$ which are independent of $N(t)$. Then 
$$ Y(t) \stackrel{law}{=} S^{k}_{A}(t).$$
\end{proposition}

\begin{proof}
Rearranging the $\phi_{S^{k}_{A}(t)}(t)$,
 $$
 \phi_{S^{k}_{A}(t)}(0) = \exp \left[(\lambda_1 + \lambda_2)kt\left(\frac{\lambda_1}{\lambda_1 + \lambda_2}\frac{1}{k}\sum_{j=1}^{k}\frac{(e^{iuj}-1)}{iuj} + \frac{\lambda_2}{\lambda_1 + \lambda_2}\frac{1}{k}\sum_{j=1}^{k}\frac{(1-e^{-iuj})}{iuj}   -1 \right)\right]
 $$
The random variables $J_{1}$ being a mixed double uniformly distributed has density 
\begin{equation}\label{pmf_X}
    p_{J_1}(x)= \sum_{i=1}^{k}p_{V_{i}}(x)f_{U_{i}}(x) = \frac{1}{k}\sum_{i=1}^{k}f_{U_{i}}(x),
\end{equation}
where $V_{j}$ follows discrete uniform distribution over $(0,k)$ with pmf
$p_{V_{j}}(x)=\mathbb{P}(V_{j}=x) = \frac{1}{k}, \;\; j=1,2,\ldots k,$
and $U_{i}$ be doubly uniform distributed random variables with density 
$$f_{U_{i}}(x) = (1-w)1_{[-i,0]}(x) + w 1_{[0, i]}(x), \;\; -i \leq x \leq i.$$
Further, $0<w<1$ is a weight parameter and $1(\cdot)$ is the indicator function. Here we obtained the characteristic of $J_1$ by using the Fourier transformation of \eqref{pmf_X},
$$\phi_{J_1} (u) =\frac{\lambda_1}{\lambda_1 + \lambda_2}\frac{1}{k}\sum_{j=1}^{k}\frac{(e^{iuj}-1)}{iuj} + \frac{\lambda_2}{\lambda_1 + \lambda_2}\frac{1}{k}\sum_{j=1}^{k}\frac{(1-e^{-iuj})}{iuj}. $$
The characteristic function of $Y(t)$ is
\begin{equation}
    \phi_{Y(t)}(u)=  e^{-k (\lambda_1 =\lambda_2) t (1-\phi_{J_1}(u))},
\end{equation}
putting the characteristic function $\phi_{J_1}(u)$ in the above expression yields the characteristic function of $S^{k}_{A}(t)$, which completes the proof.
\end{proof}
\begin{remark}
The $q$-th order moments of $J_1$ can be calculated by using \eqref{M_char} and also using Taylor series expansion of the characteristic $\phi_{J_1}(u))$, around $0$, such that
$$
\frac{(e^{iuj}-1)}{iuj} = 1+\sum_{r=1}^{\infty}\frac{(iuj)^r}{(r+1)!}\;\;\&\;\; \frac{(1-e^{-iuj})}{iuj} = 1+\sum_{r=1}^{\infty}\frac{(-iuj)^r}{(r+1)!}.
$$
We have $m_{1} = \frac{(k+1)(\lambda_1-\lambda_2)}{4(\lambda_1 +\lambda_2)}$ and $m_{2} = \frac{1}{18}[(k+1)(2k+1)]$. Further, the mean, variance and covariance of running average of SPoK are
\begin{align*}
    \mathbb{E}[S^{k}_{A}(t)]&=\mathbb{E}[N(t)]\mathbb{E}[J_1]=\frac{k(k+1)}{4}(\lambda_1 - \lambda_2)t\\
    {\rm Var}[S^{k}_{A}(t)]& = \mathbb{E}[N(t)]\mathbb{E}[J_1^2]=\frac{1}{18}[k(k+1)(2k+1)] (\lambda_1 + \lambda_2)t\\
    {\rm Cov}[S^{k}_{A}(t), S^{k}_{A}(s)] &= \frac{1}{18}[k(k+1)(2k+1)](\lambda_1-\lambda_2) s -\frac{k^2(k+1)^{2}}{16}(\lambda_1-\lambda_2)^2 s^2.
\end{align*}
\end{remark}
\begin{corollary}
For $\lambda_2 = 0$  the running average of SPoK is same as the running average of PPoK, i.e.
$$\phi_{S^k_A(t)}(u) = \phi_{N^k_A(t)}(u). $$
\end{corollary}

\begin{corollary}
For $k=1$  this process behave like the running average of Skellam process.
\end{corollary}

\begin{corollary}
The ratio of mean and variance of SPoK and running average of SPoK are $1/2$ and $1/3$ respectively.
\end{corollary}

\section{Time-changed Skellam process of order K}
We consider time-changed SPoK, which can be obtained by subordinating SPoK $S^{k}(t)$ with the independent L\'evy subordinator $D_{f}(t)$ satisfying 
$\mathbb{E}[D_{f}(t)]^c < \infty$ for all $c>0$. The time-changed SPoK is defined by
$$
Z_{f}(t) = S^{k}(D_{f}(t)),\;\; t\geq 0.
$$
Note that the stable subordinator doesn't satisfy the condition $\mathbb{E}[D_{f}(t)]^c < \infty$.
The MGF of time-changed SPoK $Z_{f}(t)$ is given by
$$
\mathbb{E}[e^{\theta Z_{f}(t)}] = e^{-t f(k(\lambda_1 + \lambda_2) -\lambda_1\sum_{j=1}^{k}e^{\theta j} -\lambda_2 \sum_{j=1}^{k}e^{-\theta j})}.
$$
\begin{theorem}
The pmf $H_{f}(t) = \mathbb{P} (Z_{f}(t) = m)$ of time-changed SPoK is given by
\begin{equation}\label{pmf_tcspok}
    H_{f}(t) = \sum_{x=\max(0, -m)}^{\infty} \frac{(k \lambda_1)^{m+x} (k \lambda_2)^{x} }{(m+x)! x!} \mathbb{E}[e^{-k(\lambda_1 +\lambda_2)D_{f}(t)} D_{f}^{2m+x}(t)],\; m \in \mathbb{Z}.
\end{equation}
\end{theorem}
\begin{proof} Let $h_{f}(x,t)$ be the probability density function of L\'evy subordinator.  Using conditional argument
\begin{align*}
    H_{f}(t) &= \int_{0}^{\infty}R_{m}(y) h_{f}(y,t) dy\\
    & = \int_{0}^{\infty}e^{-ky(\lambda_1+\lambda_2)}{\left(\frac{\lambda_1}{\lambda_2}\right)}^{m/2}I_{|m|}(2yk\sqrt{\lambda_1 \lambda_2})h_{f}(y,t) dy\\
    &=\sum_{x=\max(0, -m)}^{\infty} \frac{(k \lambda_1)^{m+x} (k \lambda_2)^{x} }{(m+x)! x!} \int_{0}^{\infty} e^{-k(\lambda_1 +\lambda_2)y} y^{2m+x} h_{f}(y,t) dy\\
    & = \sum_{x=\max(0, -m)}^{\infty} \frac{(k \lambda_1)^{m+x} (k \lambda_2)^{x} }{(m+x)! x!} \mathbb{E}[e^{-k(\lambda_1 +\lambda_2)D_{f}(t)} D_{f}^{2m+x}(t)].
\end{align*}
\end{proof}

\begin{proposition}
The state probability $H_{f}(t)$ of time-changed SPoK satisfies the normalizing condition
$$
\sum_{m=-\infty}^{\infty} H_{f}(t) =1.
$$
\end{proposition}
\begin{proof} Using \eqref{pmf_tcspok}, we have
\begin{align*}
    \sum_{m=-\infty}^{\infty} H_{f}(t)& = \sum_{m=-\infty}^{\infty}\int_{0}^{\infty}e^{-ky(\lambda_1+\lambda_2)}{\left(\frac{\lambda_1}{\lambda_2}\right)}^{m/2}I_{|m|}(2yk\sqrt{\lambda_1 \lambda_2})h_{f}(y,t) dy\\
    & = \int_{0}^{\infty} e^{-ky(\lambda_1+\lambda_2)}e^{ky(\lambda_1+\lambda_2)} h_{f}(y,t) dy\\
    & = \int_{0}^{\infty}h_{f}(y,t) dy =1.
\end{align*}
\end{proof}

\noindent The mean and covarience of time changed SPoK are given by,
\begin{align*}
    \mathbb{E}[Z_{f}(t)]&=\frac{k(k+1)}{2}(\lambda_1 -\lambda_2) \mathbb{E}[D_{f}(t)]\\
    {\rm Cov}[Z_{f}(t),Z_{f}(s)] &= \frac{1}{6}[k(k+1)(2k+1)](\lambda_1 +\lambda_2)) \mathbb{E}[D_{f}(s)]+\frac{k^2(k+1)^2}{4}(\lambda_1-\lambda_2)^2 {\rm Var}[D_{f}(s)].
\end{align*}

\section{Space fractional Skellam process and tempered space fractional Skellam process}
In this section, we introduce time-changed Skellam processes where time time-change are stable subordinator and tempered stable subordinator. These processes give the space-fractional version of the Skellam process similar to the time-fractional version of the Skellam process introduced in \cite{Kerss2014}.
\subsection{The space-fractional Skellam process}
In this section, we introduce space-fractional Skellam processes (SFSP). Further, for introduced processes, we study main results such as state probabilities and governing difference-differential equations of marginal PMF.
\begin{definition}[SFSP]
 Let $ N_{1}(t)$ and $N_{2}(t)$ be two independent homogeneous Poison processes with intensities $ \lambda_{1} >0$ and $ \lambda_{2} >0,$ respectively. Let $D_{\alpha_1}(t)$ and $D_{\alpha_2}(t)$ be two independent stable subordinators with indices $\alpha_{1} \in (0,1)$ and $\alpha_{2} \in (0,1)$ respectively. These subordinators are independent of the Poisson processes $ N_{1}(t)$ and $N_{2}(t)$. The subordinated stochastic process
$$ S_{\alpha_1, \alpha_2}(t)= N_{1}(D_{\alpha_1}(t)) - N_{2}(D_{\alpha_2}(t)) $$
is called a SFSP.
\end{definition}
\noindent Next we derive the moment generating function (MGF) of SFSP. We use the expression for marginal (PMF) of SFPP given in \eqref{space-fractional-PMF} to obtain the marginal PMF of SFSP.
\begin{align*}
M_{\theta}(t)=\mathbb{E}[e^{\theta S_{\alpha_1, \alpha_2}(t)}] &= e^{\theta(N_{1}(D_{\alpha_1}(t))- N_{2}(D_{\alpha_2}(t)))}
= e^{-t[\lambda_{1}^{\alpha_{1}}(1-e^{\theta})^{\alpha_{1}}+\lambda_{2}^{\alpha_{2}}(1-e^{-\theta})^{\alpha_{2}}]},\;\; \theta \in \mathbb{R}.
\end{align*}
%where we use the MGF of SFPP, which is given $\mathbb{E}[e^{\theta N(D_{\alpha}(t))}] = e^{-\lambda^{\alpha}(1-e^{\theta})^{\alpha}}t$.
In the next result, we obtain the state probabilities of the SFSP.
\begin{theorem}
The PMF $H_{k}(t)=\mathbb{P}(S_{\alpha_1, \alpha_2}(t)=k)$ of SFSP is given by 
\begin{align}
H_{k}(t)&=
 \sum_{n=0}^{\infty}\frac{(-1)^k}{n!(n+k)!}\left({}_{1}\psi_{1} \left[\begin{matrix}
 (1, \alpha_1); \\
  (1-n-k, \alpha_1); 
 \end{matrix}(-{\lambda_1}^{\alpha_1} t) \right]\right)
 \left({}_{1}\psi_{1} \left[\begin{matrix}
 (1, \alpha_2); \\
  (1-n, \alpha_2); 
 \end{matrix}(-{\lambda_2}^{\alpha_2} t) \right]\right)\mathbb{I}_{k\geq 0}\nonumber\\
 &+\sum_{n=0}^{\infty}\frac{(-1)^{|k|}}{n!(n+|k|)!}\left({}_{1}\psi_{1} \left[\begin{matrix}
 (1, \alpha_1); \\
  (1-n, \alpha_1); 
 \end{matrix}(-{\lambda_1}^{\alpha_1} t) \right]\right)
 \left({}_{1}\psi_{1} \left[\begin{matrix}
 (1, \alpha_2); \\
  (1-n-|k|, \alpha_2); 
 \end{matrix}(-{\lambda_2}^{\alpha_2} t) \right]\right)\mathbb{I}_{k< 0}
\end{align}
for $k\in \mathbb{Z}$.
\end{theorem}
\begin{proof} Note that $N_{1}(D_{\alpha_1}(t))$ and  $N_{2}(D_{\alpha_2}(t))$ are independent, hence
\begin{align*}
\mathbb{P}(S_{\alpha_1, \alpha_2}(t)=k) &=  \sum^{\infty}_{n=0}\mathbb{P}(N_{1}(D_{\alpha_1}(t))=n+k)\mathbb{P}(N_{2}(D_{\alpha_2}(t))=n)\mathbb{I}_{k\geq 0}\\
&+ \sum^{\infty}_{n=0}\mathbb{P}(N_{1}(D_{\alpha_1}(t))=n)\mathbb{P}(N_{2}(D_{\alpha_2}(t))=n+|k|)\mathbb{I}_{k<0}.
\end{align*}
Using \eqref{space-fractional-PMF}, the result follows.
\end{proof}

\noindent In the next theorem, we discuss the governing differential-difference equation of the marginal PMF of SFSP.  
\begin{theorem}
The marginal distribution $H_{k}(t)= \mathbb{P}(S_{\alpha_1, \alpha_2}(t)=k)$ of SFSP satisfy  the following differential difference equations
\begin{align}
 \frac{d}{dt}H_{k}(t) &= -\lambda_{1}^{\alpha_{1}} (1-B)^{\alpha_{1}} H_{k}(t)-\lambda_{2}^{\alpha_{2}} (1-F)^{\alpha_{2}} H_{k}(t),\;\;   k \in \mathbb{Z} \\
\frac{d}{dt}H_{0}(t)& = -\lambda_{1}^{\alpha_{1}}H_{0}(t)-\lambda_{2}^{\alpha_{2}}H_{1}(t),
\end{align}
with initial conditions $H_{0}(0)=1$ and $H_{k}(0)=0$ for $k\neq0.$
\end{theorem}
\begin{proof}
The proof follows by using probability generating function.
\end{proof}
\begin{remark}
The MGF of the SFSP solves the differential equation
\begin{align}
\frac{dM_{\theta}(t)}{dt} = -M_{\theta}(t)(\lambda_{1}^{\alpha_{1}}(1-e^{\theta})^{\alpha_{1}}+\lambda_{2}^{\alpha_{2}}(1-e^{-\theta})^{\alpha_{2}}).
\end{align}
\end{remark}

\begin{proposition}
The L\'evy density $\nu_{S_{\alpha_1, \alpha_2}}(x)$ of SFSP is given by
\begin{align*}\label{skellam_levy-I}
\nu_{S_{\alpha_1, \alpha_2}} (x) = {\lambda_1}^{\alpha_1}\sum^{\infty}_{n_1=1}(-1)^{n_1+1} {\alpha_1 \choose n_1} \delta_{n_1}(x)+\lambda_2^{\alpha_2}\sum^{\infty}_{n_2=1}(-1)^{n_2+1} {\alpha_2 \choose n_2} \delta_{-n_2}(x).
\end{align*}
\end{proposition}
\begin{proof} Substituting the L\'evy densities $\nu_{N_{\alpha_1}}(x)$ and $\nu_{N_{\alpha_2}}(x)$ of $N_{1}(D_{\alpha_1}(t))$ and $N_{2}(D_{\alpha_2}(t))$, respectively from the equation \eqref{space_levy}, we obtain
 $$\nu_{S_{\alpha_1, \alpha_2}} (x) = \nu_{N_{\alpha_1}}(x) +  \nu_{N_{\alpha_2}}(x),$$
which gives the desired result.
\end{proof}

\subsection{Tempered space-fractional Skellam process (TSFSP)}
In this section, we present the tempered space-fractional Skellam process (TSFSP). We discuss the corresponding fractional difference-differential equations, marginal PMFs and moments of this process.

\begin{definition}[TSFSP]
The TSFSP is obtained by taking the difference of two independent tempered space fractional Poisson processes. Let $D_{\alpha_1, \mu_1}(t)$, $D_{\alpha_2, \mu_2}(t)$ be two independent TSS (see \cite{Rosinski2007}) and $N_1(t), N_2(t)$ be two independent Poisson processes whcih are independent of TSS. Then the stochastic process 
$$ S^{\mu1, \mu2}_{\alpha_1,\alpha_2}(t) = N_1(D_{\alpha_1, \mu_1}(t)-N_2(D_{\alpha_2, \mu_2}(t))$$
is called the TSFSP.
\end{definition}

\begin{theorem}
The PMF $H_{k}^{\mu1, \mu2}(t)=\mathbb{P}(S^{\mu1, \mu2}_{\alpha_1,\alpha_2}(t)=k)$ is given by
\begin{align}
H_{k}^{\mu1, \mu2}(t)&=\sum_{n=0}^{\infty}\frac{(-1)^{k}}{n!(n+k)! }e^{t(\mu_{1}^{\alpha_{1}}+\mu1^{\alpha_{1}})}\left(\sum_{m=0}^{\infty}\frac{\mu_1^m \lambda_1^{-m}}{m!} {}_{1}\psi_{1} \left[\begin{matrix}
 (1, \alpha_1); \\
  (1-n-k-m, \alpha_1); 
 \end{matrix}(-{\lambda_1}^{\alpha_1} t) \right]\right)\times\nonumber\\
& \left(\sum^{\infty}_{l=0}\frac{{\mu_2}^{l}{\lambda_2}^{-l}}{l!}{}_{1}\psi_{1} \left[\begin{matrix}
 (1, \alpha_2); \\
  (1-l-k, \alpha_2); 
 \end{matrix}(-{\lambda_2}^{\alpha_2} t) \right]\right)
\end{align}
when $k\geq 0$ and similarly for $k<0$,
\begin{align}
H_{k}^{\mu1, \mu2}(t)&=\sum_{n=0}^{\infty}\frac{(-1)^{|k|}}{n!(n+|k|)! }e^{t(\mu_{1}^{\alpha_{1}}+\mu1^{\alpha_{1}})}\left(\sum_{m=0}^{\infty}\frac{\mu_1^m \lambda_1^{-m}}{m!} {}_{1}\psi_{1} \left[\begin{matrix}
 (1, \alpha_1); \\
  (1-n-m, \alpha_1); 
 \end{matrix}(-{\lambda_1}^{\alpha_1} t) \right]\right)\times\nonumber\\
& \left(\sum^{\infty}_{l=0}\frac{{\mu_2}^{l}{\lambda_2}^{-l}}{l!}{}_{1}\psi_{1} \left[\begin{matrix}
 (1, \alpha_2); \\
  (1-l-n-|k|, \alpha_2); 
 \end{matrix}(-{\lambda_2}^{\alpha_2} t) \right]\right).
\end{align}
\end{theorem}

\begin{proof}
Since $N_{1}(D_{\alpha_1, \mu_1}(t))$ and $N_{2}(D_{\alpha_2, \mu_2}(t))$ are independent,
\begin{align*}
\mathbb{P}\left(S^{\mu1, \mu2}_{\alpha_1,\alpha_2}(t)=k\right) &=  \sum^{\infty}_{n=0}\mathbb{P}(N_1(D_{\alpha_1, \mu_1}(t))=n+k)\mathbb{P}(N_2(D_{\alpha_2, \mu_2}(t))=n)\mathbb{I}_{k\geq 0}\\
&+ \sum^{\infty}_{n=0}\mathbb{P}(N_1(D_{\alpha_1, \mu_1}(t))=n)\mathbb{P}(N_2(D_{\alpha_2, \mu_2}(t))=n+|k|)\mathbb{I}_{k<0},
\end{align*}
which gives the marginal PMF of TSFPP by using \eqref{pmf_tem_space}. 
\end{proof}

\begin{remark}
We use this expression to calculate the marginal distribution of TSFSP. The MGF is obtained by using the conditioning argument. Let $f_{\alpha, \mu}(x,t)$ be the density function of $D_{\alpha, \mu}(t)$. Then
\begin{align}\label{MGF_temp}
\mathbb{E}[e^{\theta N(D_{\alpha, \mu}(t))}] &= \int^{\infty}_{0}\mathbb{E}[e^{\theta N(u)}]f_{\alpha, \mu}(u,t)du
 =e^{-t\{(\lambda(1-e^{\theta})+\mu)^{\alpha}-\mu^{\alpha}\}}.
\end{align}
Using \eqref{MGF_temp}, the MGF of TSFSP is 
$$
 \mathbb{E}[e^{\theta S^{\mu_1, \mu_2}_{\alpha_1, \alpha_2}(t)}]=\mathbb{E}\left[e^{\theta N_{1}(D_{\alpha_1, \mu_1}(t))}\right]\mathbb{E}\left[e^{\theta N_{2}(D_{\alpha_2, \mu_2}(t))}\right] = e^{-t[\{(\lambda_1(1-e^{\theta})+\mu_1)^{\alpha_1}-\mu_1^{\alpha_1}\}+\{(\lambda_2(1-e^{\theta})+\mu_2)^{\alpha_2}-\mu_2^{\alpha_2}\}]}.
$$
\end{remark}

\begin{remark}
We have $\mathbb{E}[S^{\mu1, \mu2}_{\alpha_1,\alpha_2}(t)] =t(\alpha_1 \mu_1^{\alpha_1 -1}-\alpha_2 \mu_2^{\alpha_2 -1}).$ Further, the covariance of TSFSP can be obtained by using \eqref{vari_tem} and 
\begin{align*}
{\rm Cov}\left[S^{\mu_1, \mu_2}_{\alpha_1, \alpha_2}(t), S^{\mu_1, \mu_2}_{\alpha_1, \alpha_2}(s)\right] & =
{\rm Cov}[N_{1}(D_{\alpha_1, \mu_1}(t)), N_{1}(D_{\alpha_1, \mu_1}(s))]
+ {\rm Cov}[N_{2}(D_{\alpha_2, \mu_2}(t)), N_{2}(D_{\alpha_2,\mu_2}(s))]\\
&= {\rm Var}(N_{1}(D_{\alpha_1, \mu_1}(\min(t,s)))+ {\rm Var}(N_{2}(D_{\alpha_2, \mu_2}(\min(t,s))).
\end{align*}
\end{remark}

\begin{proposition}
The L\'evy density $\nu_{S^{\mu1, \mu2}_{\alpha_1,\alpha_2}}(x)$ of TSFSP is given by
\begin{align*}\label{skellam_levy-I}
\nu_{S^{\mu1, \mu2}_{\alpha_1,\alpha_2}}(x) &= \sum_{n_1=1}^{\infty}\mu_1^{\alpha_1-n_1}{\alpha_1 \choose n_1}{\lambda_1}^{n_1} \sum_{l=1}^{n_1} {n_1 \choose l_1}(-1)^{l_1+1} \delta_{l_1}(x)\\
& +\sum_{n_2=1}^{\infty}{\mu_2}^{\alpha_2-n_2}{\alpha_2 \choose n_2}{\lambda_2}^{n_2} \sum_{l_2=1}^{n_2} {n_2 \choose l_2}(-1)^{l_2+1} \delta_{l_2}(x), \; \mu_{1}, \; \mu_{2} > 0.
\end{align*}
\end{proposition}

\begin{proof} By adding L\'evy densities $\nu_{N_{\alpha_1, \mu_1}}(x)$ and $\nu_{N_{\alpha_2, \mu_2}}(x)$ of $N_{1}(D_{\alpha_1, \mu_1}(t))$ and $N_{2}(D_{\alpha_2, \mu_2}(t))$ respectively from the equation \eqref{levy_temp}, which leads to

 $$\nu_{S^{\mu1, \mu2}_{\alpha_1,\alpha_2}}(x) = \nu_{N_{\alpha_1, \mu_1}}(x) + \nu_{N_{\alpha_2, \mu_2}}(x).$$
\end{proof}

\subsection{Simulation of SFSP and TSFSP}
We present the algorithm to simulate the sample trajectories for SFSP and TSFSP. We use Python 3.7 and its libraries \textit{Numpy} and \textit{ Matplotlib} for the simulation purpose.\\

\noindent{\bf Simulation of SFSP:}\\
\textbf{Step-1}: generate independent and uniformly distributed in $[0, 1]$ rvs $U$, $V$ for fix values of parameters;\\
\textbf{Step-2}: generate the increments of the $\alpha$-stable subordinator $D_{\alpha}(t)$ (see \cite{Cahoy2010}) with pdf $f_{\alpha}(x, t)$, using the relationship $D_{\alpha}(t + dt) - D_{\alpha}(t) \stackrel{d}{=} D_{\alpha}(dt) \stackrel{d}= (dt)^\frac{1}{\alpha} D_{\alpha}(1)$, where
$$   
D_{\alpha}(1) = \frac{\sin(\alpha \pi U)[\sin((1-\alpha)\pi U)]^{1/\alpha -1}}{[\sin(\pi U)]^{1/\alpha} |\log V|^{1/\alpha -1}};
$$
\textbf{Step-3}: generate the increments of Poisson distributed rv $N(D_{\alpha}(dt))$ with parameter $\lambda (dt)^{1/\alpha} D_{\alpha}(1)$;\\
\textbf{Step-4}:  cumulative sum of increments gives the space fractional Poisson process $N(D_{\alpha}(t))$  sample trajectories;\\
\textbf{Step-5}: generate $N_{1}(D_{\alpha_1}(t))$, $N_{2}(D_{\alpha_2}(t))$ and subtract these to get the SFSP $S_{\alpha_1, \alpha_2}(t)$.\\

\noindent We next present the algorithm for generating the sample trajectories of TSFSP.\\

\noindent{\bf Simulation of TSFSP:}\\
Use the first two steps of previous algorithm for generating the increments of $\alpha$-stable subordinator $D_{\alpha}(t)$.\\
\textbf{Step-3}: for generating the increments of TSS $D_{\alpha, \mu}(t)$ with pdf $f_{\alpha, \mu}(x, t)$, we use the following steps called ``acceptance-rejection method";\\
\textbf{(a)} generate the stable random variable $D_{\alpha}(dt)$;\\
\textbf{(b)} generate uniform $(0, 1)$ rv $W$ (independent from $D_{\alpha}$);\\
\textbf{(c)} if $W \leq e^{-\mu D_{\alpha}(dt)}$, then $D_{\alpha, \mu}(dt) = D_{\alpha}(dt)$ (``accept"); otherwise go back to (a) (``reject").
Note that, here we used that $f_{\alpha, \mu}(x, t) = e^{-\mu x + \mu^{\alpha}t} f_{\alpha}(x, t),$ which implies $\frac{f_{\alpha, \mu}(x, t)(x,dt)}{cf_{\alpha}(x,dt)} = e^{-\mu x}$ for $c = e^{{\mu }^{\alpha} dt}$ and the ratio is bounded between $0$ and $1$;\\
\textbf{step-4}: generate Poisson distributed rv $N(D_{\alpha, \mu}(dt))$ with parameter $\lambda D_{\alpha, \mu}(dt)$ \\
\textbf{step-5}: cumulative sum of increments gives the tempered space fractional Poisson process $N(D_{\alpha, \mu}(t))$ sample trajectories;\\
\textbf{step-6}: generate $N_{1}(D_{\alpha_1, \mu_1}(t))$, $N_{2}(D_{\alpha_2, \mu_2}(t))$, then take difference of these to get  the sample paths of the TSFSP$S_{\alpha_1, \alpha_2}^{\mu_1, \mu_2}(t)$.

\begin{figure}[ht!]
        \centering
        \includegraphics[width=0.45\textwidth, height=0.3\textheight]{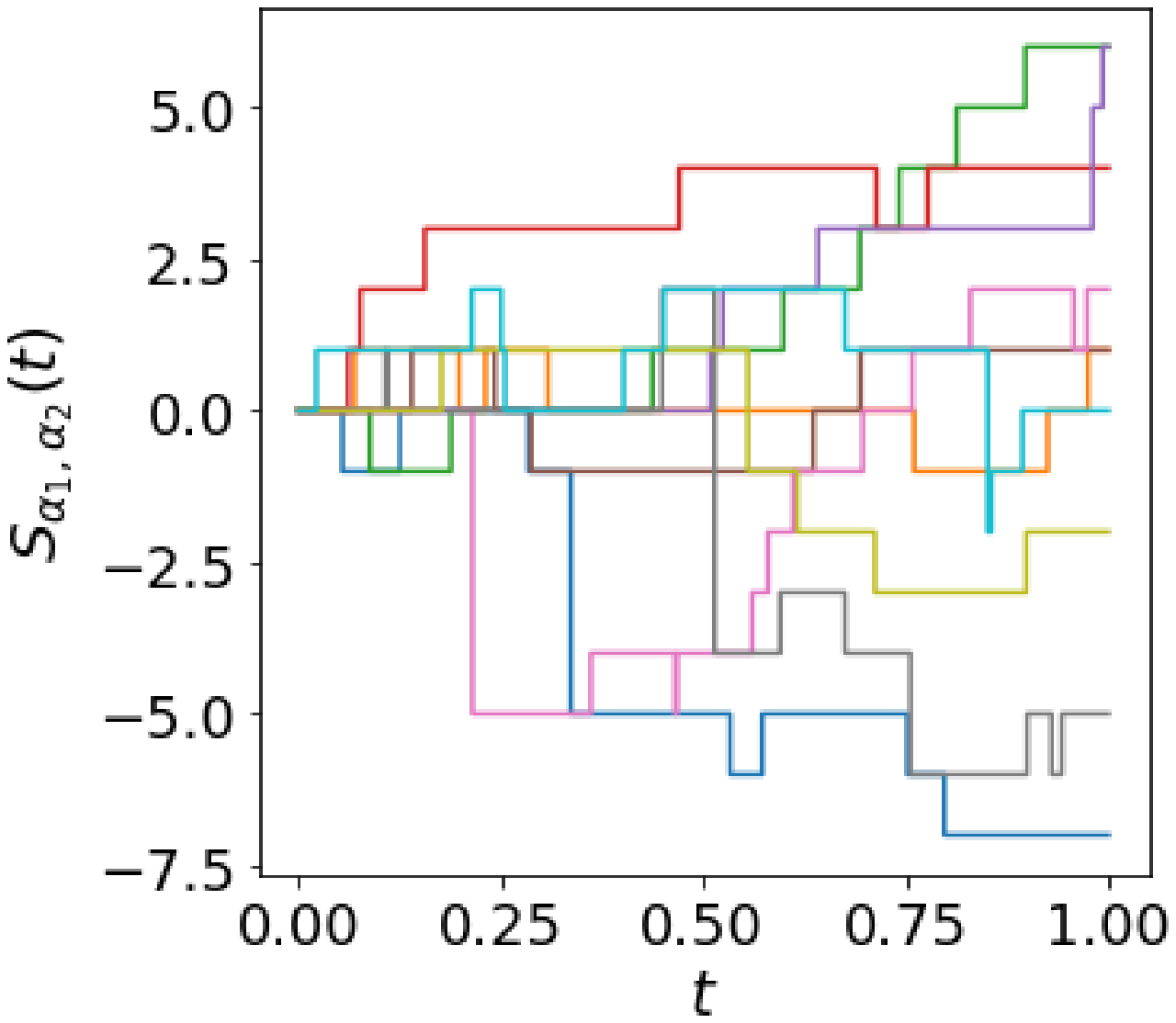}
       \includegraphics[width=0.45\textwidth, height=0.3\textheight]{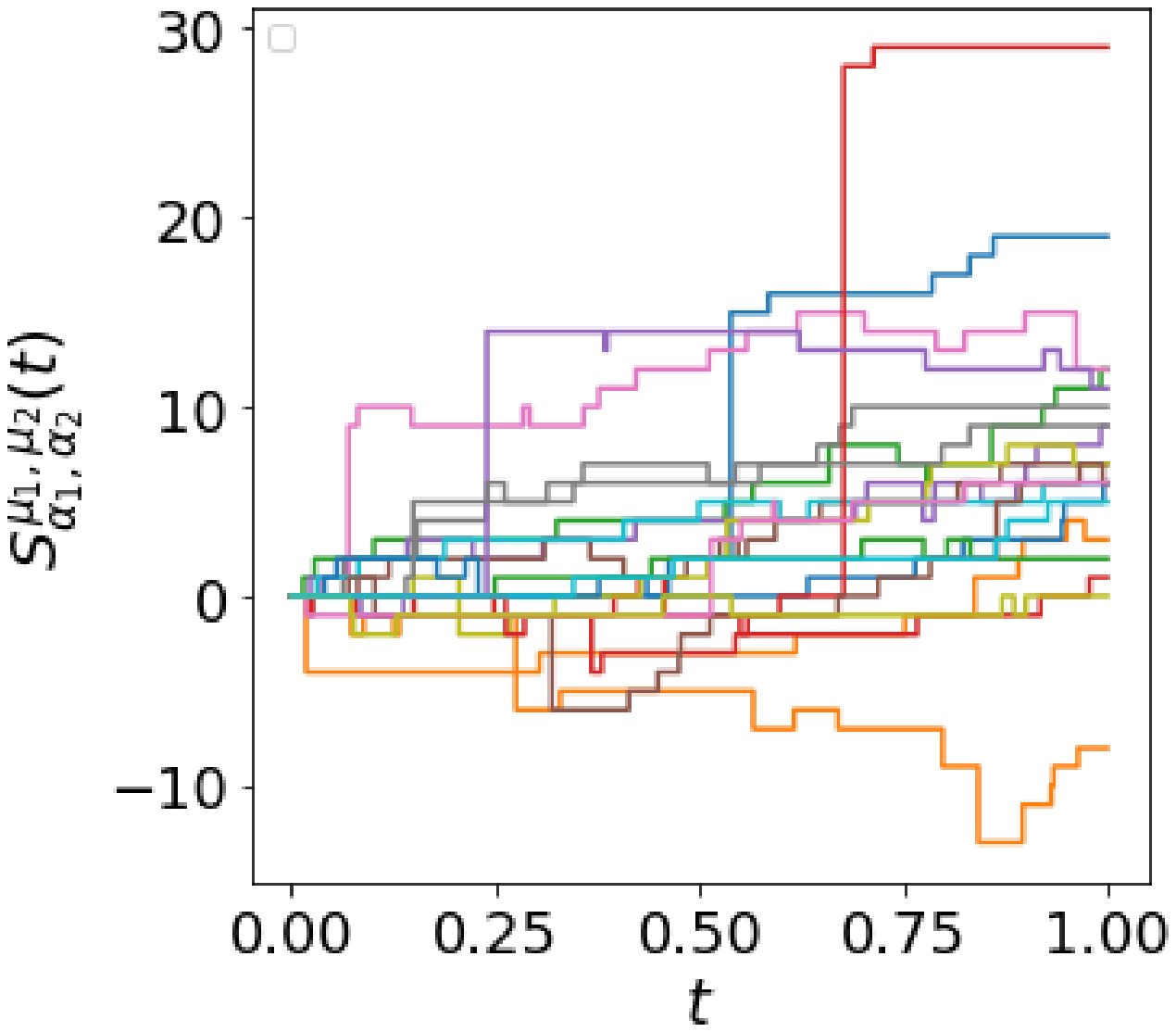}
			\caption{The sample trajectories of SFSP and TSFSP}
\end{figure}

\noindent {\bf Acknowledgments:} 
NG would like to thank Council of Scientific and Industrial Research(CSIR), India, for the award of a research fellowship.  
%N. Leonenko was supported in particular by Australian Research Council's Discovery Projects funding scheme (project DP160101366)and  by project MTM2015-71839-P of MINECO, Spain (co-funded with FEDER funds).

%\vspace{2cm}
%\noindent {\bf \Large References}
\noindent
%\begin{namelist}{xxx}

%\end{namelist}                       
\end{document}